\documentclass[12pt,a4paper]{amsart}

\usepackage{enumitem}
\usepackage{hyperref}

\usepackage{amsmath}

 \newtheorem{theor}{Theorem}[section]
 \newtheorem{corol}[theor]{Corollary}
 \newtheorem{lemma}[theor]{Lemma}   
 \newtheorem{prop}[theor]{Proposition}
 \newtheorem{example}[theor]{Example}
 \newtheorem{claim}[theor]{Claim}
 \newtheorem{dfn}[theor]{Definition}
 
 \theoremstyle{remark}
 \newtheorem*{remark}{Remark}

 \newtheorem{question}[theor]{\sc Problem}

\newcommand{\real }{{\mathbb R}}
\newcommand{\complex }{{\mathbb C}}

\newcommand{\semi}{{\real_{2}^{4}}}

\newcommand {\lpr}[2]{{\langle}{#1},{#2}{\rangle}}

\newcommand {\D}{{\mathcal D}}

\DeclareMathOperator{\Span}{span}

\title[Minimal Spacelike Surfaces in $\semi$]{A Cauchy Problem for
  minimal spacelike surfaces in $\semi$}
\author[A. Lymberopoulos, A. P. Franco Filho and
A. E. Paternina]{Alexandre Lymberopoulos\and Ant\^onio de P\'adua Franco
  Filho\and Anuar Enrique Paternina Montalvo}

\address{Departamento de Matem\'atica - Universidade de S\~ao Paulo}
\email{lymber@ime.usp.br, apadua@ime.usp.br, anuarep@ime.usp.br}

\begin{document}

\begin{abstract}
  We give a definition of isoclinic parametric surfaces in $\semi$ and
  prove that such an isoclinic conformal immersion comes from two
  holomorphic functions. A Cauchy problem is proposed and solved,
  namely: construct an isoclinic and minimal positive (negative)
  spacelike surface in $\semi$, containing a given positive (negative)
  real analytic curve. At last, we study the important and well-known
  Björling problem, illustrating with some examples given in the last
  section.
\end{abstract}

\maketitle

\section{Introduction}\label{sec:intro}

In this work we study spacelike surfaces in $\semi$ aiming to solve
well-known classical problems. Given any two planes $P^2$ and
$\overline{P}^2$ in $\mathbb{R}^4$, we consider the angle between a unit
vector $v$ in $P^2$ and its orthogonal projection $\bar{v}$ in
$\overline{P}^2$. When $v$ describes a circumference of radius $1$
centred at the origin, that angle varies, in general, between two
distinct extreme values or, equivalently, that unit circumference in
$P^2$ projects as an ellipse in $\overline{P}^2$ with axes corresponding
to the extreme values of the above angle.

In \cite{Won}, Wong developed a curvature theory for surfaces in
$\mathbb{R}^4$ based on angles between two tangent planes of the
surface. When the angle remains constant, i.e., the ellipse is also a
circumference, we say that the planes are isoclinic to each other. An
interesting connection with functions of one complex variable is the
well-known theorem establishing that a 2-dimensional surface of class
$C^2$ in $\mathbb{R}^4$ has the property that all of its tangent
2-planes are mutually isoclinic if, and only if the surface is a
$R$-surface, that is, a surface given in suitable rectangular
coordinates $(x,y,u,v)$ in $\mathbb{R}^4$ by $u=u(x,y)$, $v=v(x,y)$,
where $u(x,y)$ and $v(x,y)$ are the real and imaginary parts of a
complex analytic function $f(x+iy)$. In the higher dimensional case,
Wong shows in \cite{Wog} that the only $n$-dimensional surfaces of class
$C^2$ in $\mathbb{R}^{2n}$ $(n>2)$ whose tangent $n$-planes are all
mutually isoclinic are the $n$-planes.

In the same way, we can consider two planes positive (negative) planes
in $\semi$ and take the angle that a unit vector in one of those planes
makes with its orthogonal projection in the other plane, depending on
causal character of the plane spanned by those vectors. In
Definition\,\ref{cli} we define positive (negative) planes in $\semi$ to
be isoclinic to each other, and using the operator $L$ given by
(\ref{trans}) we provide a characterization of such planes. Then we
present a definition of isoclinic parametric surfaces in $\semi$, which
will be objects of study of our work.

In Section \ref{sec:cauchyp}, we propose and solve a Cauchy problem
which asks about the existence of an isoclinic and minimal positive
(negative) spacelike surface in $\semi$ containing a given positive
(negative) real analytic curve.

In Section \ref{sec:bjorlingp}, we deal with the well-known Bj\"orling
problem, in this case for positive (negative) spacelike surface in
$\semi$. It consists of constructing a minimal positive (negative)
spacelike surface in $\semi$ containing a given real analytic
strip. This problem was proposed in Euclidean space $\mathbb{R}^3$ by
Bj\"orling in 1844 and its solution obtained by Schwarz in 1890 through
an explicit formula in terms of initial data. Thereafter, the Bj\"orling
problem has been considered in other ambient spaces, including in larger
codimension or with indefinite metrics. Some works in the literature are
\cite{ACM, AM, AV, CDM, DFM, DM, MO}.

\section{Some algebraic and analytic preliminary facts}\label{sec:algfacts}

The $4$-dimensional pseudo-Euclidean space $\real_{2}^{4}$ is the
$4$-dimensional vector space $\real^{4}$ equipped with the
pseudo-Riemannian metric
\begin{equation}\label{eq:semimetric}
  ds^{2} = - (d x^{1})^{2} - (d x^{2})^{2} + (d x^{3})^{2} + (d x^{4})^{2},
\end{equation} 
oriented by the volume form
\[\omega = d x^{1} \wedge d x^{2} \wedge d x^{3} \wedge d x^{4}.\]

The inner product associated to the quadratic form $ds^{2}$ is given by
\begin{equation}\label{eq:semiorid}
\lpr{v}{w} = - v^{1}w^{1} - v^{2}w^{2} + v^{3}w^{3} + v^{4}w^{4}.
\end{equation}  

The elements of the canonical basis of the vector space $\real^4$ will
be denoted by $e_{i}$, $i=1,2,3,4$. The six coordinate planes
$\pi_{ij}=\Span\{e_{i},e_{j}\}$, for $i,j = 1,2,3,4$, are such that
$\pi_{12}$ is a negative Euclidean plane, this means that
$-ds^{2}(\pi_{12})=(dx^1)^2 + (dx^2)^2$ and thus it has a negative
definite induced metric; $\pi_{34}$ is a positive Euclidean plane, since
$ds^{2}(\pi_{34})=(dx^3)^2 + (dx^4)^2$; whereas the other four
coordinate planes $\pi_{13}$, $\pi_{14}$, $\pi_{23}$ and $\pi_{24}$ are
Lorentzian planes, for example $ds^{2}(\pi_{13})=-(dx^1)^2 + (dx^3)^2$.

A positive (respec. negative) spacelike plane is a $2$-dimensional
subspace $V^{2} \subset \real_{2}^{4}$ such that the induced metric
$ds^{2}(V^{2})$ is positive (respec. negative) definite. Given a
positive (respec. negative) spacelike plane $V^2$, we say that a basis
$\{v_{1},v_{2}\}$ for $V^2$ is a $\epsilon \lambda$-isothermic basis if,
and only if,
\begin{equation}
  \lpr{v_{i}}{v_{j}} = \epsilon \lambda^{2} \delta_{ij}\mbox{, for $i,j
    = 1,2$ and $\epsilon = 1$ (respec. $\epsilon =-1$)}.
\end{equation} 

The geometry of a given negative plane $V^{2}$ with an orthonormal basis
$\{v_{1},v_{2}\}$ is Euclidean, since the induced metric is
\[-ds^{2}(v_{i},v_{j}) = -\lpr{v_{i}}{v_{j}} = \delta_{ij}.\], 

If $\{v,w\}$ is a $(-1)$-orthonormal set, we say that the negative plane
$V^{2}=\Span\{v,w\}$ has positive induced orientation if, and only
if, the projection from $\semi$ onto $\pi_{12}$ defined by
\[pr_{12}(v^{1},v^{2},v^{3},v^{4}) = v^{1} e_{1} + v^{2} e_{2}\] give us
a positive basis relative to $\{e_{1},e_{2}\}$, in this
order. Analogously, we define the positive induced orientation for
positive planes with the projection $pr_{34}$. If it is given a negative
plane $V^{2} =\Span\{v_{1},v_{2}\}$ and a positive plane
$W^{2} =\Span\{w_{1},w_{2}\}$, both positively oriented, such that
the set $\mathcal{B}=\{v_{1},v_{2},w_{1},w_{2}\}$ is an orthonormal set,
then $\mathcal{B}$ is a positive orthonormal basis of $\real_2^4$.

When necessary, we will use the following notation for the projections
onto the planes $\pi_{12}$ and $\pi_{34}$:
\[pr_{12}(v) = \hat{v}\quad\mbox{and}\quad pr_{34}(v) = \tilde{v},\]
thus $v = \hat{v} + \tilde{v}$.

Two useful elements of the orthogonal group of $\semi$ are given, in
matrix form, by
\begin{equation}\label{trans}
  L = \begin{bmatrix}
    0 &-1 & 0 & 0 \\ 1 & 0 & 0 & 0 \\ 
    0 & 0 & 0 &-1 \\ 0 & 0 & 1 & 0
  \end{bmatrix}\quad
  \mbox{and}\quad
  N = \begin{bmatrix}
    0 & 0 & -1 & 0 \\ 0 & 0 & 0 & 1 \\ 
    -1 & 0 & 0 & 0 \\ 0 & 1 & 0 & 0
  \end{bmatrix}.
\end{equation}
Notice that $\det L =\det N=1$, $-L^2 = N^2 = I$ and $LN + NL = 0$,
where $I$ is the identity matrix.

If $v=(v^1,v^2,v^3,v^4)$ is a positive (respec. negative) vector of
$\semi$, then $L(v)=(-v^2,v^1,-v^4,v^3)$ is another positive
(respec. negative) vector such that $\left\lbrace v,L(v)\right\rbrace$
is a $\epsilon\vert v\vert$-isothermic basis for the positive
(respec. negative) plane spanned by these vectors. We have an analogous
statement for the transformaton $N$.

\subsection{Spheres in $\semi$}

The positive ($\epsilon = 1$) sphere of radius $r > 0$, the negative
($\epsilon = -1$) sphere of radius $r > 0$ and the null ($\epsilon = 0$)
null sphere in $\semi$ are defined by
\[S_{2}^{3}(\epsilon r) = \{v \in \semi \colon  \lpr{v}{v} = \epsilon r^{2}
  \}.\] We give a parametrization for unit spheres ($r=1$) as follows:
 
If $w = (a,b,x,y) \in S_{2}^{3}(-1)$, then
$a^{2} + b^{2} = 1 + x^{2} + y^{2}$ so we can take the parametric hypersurface
\[F(\varphi,\theta, \eta) = (\cosh \varphi \cos \theta, \cosh \varphi
  \sin \theta, \sinh \varphi \cos \eta, \sinh \varphi \sin \eta), \]
satisfying $F(\real \times [0, 2\pi]^{2}) = S_{2}^{3}(-1)$.

In the same way, if $w = (a,b,x,y) \in S_{2}^{3}(1)$, then
$a^{2} + b^{2} + 1 = x^{2} + y^{2}$ and 
\[G(\varphi,\theta, \eta) = (\sinh \varphi \cos \theta, \sinh \varphi
  \sin \theta, \cosh \varphi \cos \eta, \cosh \varphi \sin \eta),\] is a
parametrization such that
$G(\real \times [0, 2\pi]^{2}) = S_{2}^{3}(1)$.

For the null sphere $S_{2}^{3}(0)$, if $w = (a,b,x,y) \in S_{2}^{3}(0)$,
then $a^{2} + b^{2} = x^{2} + y^{2}$ and we take the parametric
hypersurface
\[H(\varphi,\theta, \eta) = (e^{\varphi} \cos \theta, e^{\varphi} \sin
  \theta, e^{\varphi} \cos \eta, e^{\varphi} \sin \eta),\] such that
$H(\real \times [0, 2\pi]^{2}) = S_{2}^{3}(0)$.

Let $pr_{12}(w) = \hat{w}$ and $pr_{34}(w) = \tilde{w}$ be,
respectively, the orthogonal projections of a vector $w \in \semi$ onto
the coordinate planes $\pi_{12}$ and $\pi_{34}$. In terms of the above
parametrizations, if $w \in S_{2}^{3}(-1)$ then
$\lpr{\hat{w}}{w} = -\cosh^2 \varphi$, and if $w \in S_{2}^{3}(1)$ then
$\lpr{\tilde{w}}{w} = \cosh^2 \varphi$. The following proposition is the
key for our definition of isoclinic planes.

\begin{prop}
  Let $\{ v,w \}$ be a $(-1)$-isothermic basis of a negative spacelike
  plane $V^{2}$ in $\semi$. The orthogonal projection of the negative
  Euclidean circumference
  $x(\theta) = \cos \theta \; v + \sin \theta \; w$ onto the plane
  $\pi_{12}$ is a negative Euclidean circumference of radius $r$ if and
  only if the hyperbolic angle function $\varphi(\theta)$ given by
  \begin{equation}
    \varphi\colon \theta \to \lpr{\hat{x}(\theta)}{x(\theta)} = - \cosh^2 \varphi(\theta)
  \end{equation}      
  is the constant $-r^2$.
\end{prop}

\begin{proof}
  Let $E$, $F$ and $G$ defined by $E = -\lpr{\hat{v}}{\hat{v}}$,
  $F = -\lpr{\hat{v}}{\hat{w}}$ and $G = -\lpr{\hat{w}}{\hat{w}}$. Then
  $\cosh^2 \varphi(\theta) = E \cos^{2} \theta + 2F \cos \theta \sin
  \theta + G \sin^{2} \theta$ and $\pi_{12}(x(\theta))$ is a
  circumference if, and only if $E = G$ and $F = 0$, leading to
  $E= \cosh^2 \varphi(\theta) = r^2$. Moreover, $\varphi \neq 0$ implies
  $V^2\cap \pi_{12}=\left\lbrace 0\right\rbrace$ and $\varphi=0$ implies
  $V^2=\pi_{12}$.
\end{proof}

Note that we have an analogous result if we consider positive spacelike
planes in $\semi$ and the orthogonal projection onto the plane
$\pi_{34}$. In this case, we deal with the Lorentzian spacelike angle
between two positive vectors than span a timelike plane (see
\cite{Rat}).

\begin{dfn}\label{cli}
  Let $V^2$ be a negative spacelike plane in $\semi$. Consider a
  $(-1)$-isothermic basis of $V^2$. We say that the negative planes
  $V^{2}$ and $\pi_{12}$ are isoclinic to each other if, and only if the
  hyperbolic angle $\varphi$ between $V^2$ and $\pi_{12}$ is
  constant. In a similar way we define isoclinicness between a positive
  spacelike plane $W^{2}$ and $\pi_{34}$.

  A parametric surface $f\colon M\to \semi$ is an \emph{isoclinic surface} if
  all of its tangent planes are positive (respec. negative) spacelike
  planes and isoclinic to $\pi_{34}$ (respec. $\pi_{12}$).
\end{dfn}

The following proposition characterizes the isoclinic planes using the
transformation $L$ given in (\ref{trans}).

\begin{prop}\label{apl}
  Let $L \colon  \semi \to \semi$ be the linear transformation given in
  (\ref{trans}). Then:
  \begin{enumerate}
    \item $x \in S_{2}^{3}(-1)$ if, and only if $L(x) \in S_{2}^{3}(-1)$
    and, $x \in S_{2}^{3}(1)$ if, and only if $L(x) \in S_{2}^{3}(1)$.
    \item If $x \in S_{2}^{3}(-1)$, then $\Span\{x,L(x)\}$ is a
    negative plane isoclinic to $\pi_{12}$.
    \item If $x \in S_{2}^{3}(1)$, then $\Span\{x,L(x)\}$ is a
    positive plane isoclinic to $\pi_{34}$.
  \end{enumerate}
  Reciprocally, if the negative planes $V^{2} =\Span\{x,y\}$ and
  $\pi_{12}$ in $\semi$ are isoclinic to each other, then $y =
  L(x)$. Analogously, if the positive planes $W^2=\Span\{x,y\}$ and
  $\pi_{34}$ in $\semi$ are isoclinic to each other, then $y=L(x)$.
\end{prop}

\begin{proof}
  We only need to show the reciprocal. For this, if
  $\left\lbrace \hat{x},\hat{y}\right\rbrace$ is a
  $(-\cosh \varphi)$-isothermic basis for $\pi_{12}$, where
  $\hat{x}= x^{1}e_{1} + x^{2}e_{2}$ and
  $\hat{y}= y^{1}e_{1} + y^{2}e_{2}$, then $y^{1} = -x^{2}$ and
  $y^{2} = x^{1}$. To preserve orientation and $(-1)$-isothermality, we
  must have $y^{3} = -x^{4}$ and $y^{4} = x^{3}$.
\end{proof}

Now, using the transformation $N$, also given in (\ref{trans}), it
follow the:

\begin{corol}
  Let $V^2=\Span\{v_1,v_2\}$ be a negative spacelike plane in $\semi$
  isoclinic to $\pi_{12}$. Then, the orthogonal complement of $V^2$ in
  $\semi$ is the positive spacelike plane $\Span\{N(v_1),N(v_2)\}$. It
  is isoclinic to $\pi_{34}$, and the Lorentzian spacelike angle between
  them equal to the hyperbolic angle between $V^2$ and $\pi_{12}$.
\end{corol}

Our first example of an isoclinic surface in $\semi$ is the following:

\begin{example}\label{exa}
  Let $f \colon  \real^{2} \longrightarrow \semi$ be the parametric surface
  given by \[f(u,v) = (u,v,(u^{2} - v^{2})/2,uv).\]

  If $U = \{w = u + iv \in \complex \colon  \vert w \vert < 1\}$ and
  $\Omega=\complex \setminus \overline{U}$, then: 

  \begin{enumerate}
    \item $S_{-}=f(U)$ and is a $2$-dimensional negative submanifold of
    $\semi$, with induced metric $-ds^{2} = (1 - u^{2} - v^{2}) du dv$.
    \item $S_{+}=f(\Omega)$ is a $2$-dimensional positive submanifold of
    $\semi$, with induced metric $ds^{2} = (1 -u^{2} - v^{2}) du dv$.
  \end{enumerate}

  Both surfaces are non-flat (non-vanishing Gauss curvature) and
  isoclinic ones.
\end{example}

In fact, for $w\in U\cup\Omega$, note that
\[E = \lpr{f_{u}}{f_{u}} = -1 + u^{2} + v^{2}=\lpr{f_{v}}{f_{v}} =
  G\quad\mbox{and}\quad F=\lpr{f_{u}}{f_{v}}= 0.\] Consider the
$(\pm1)$-isothermic negative basis $\left\lbrace a,b\right\rbrace$ of
$T_{f(w)}S_{\pm}$, where $a =\frac{1}{\sqrt{\pm E(w)}}f_{u}(w)$ and
$b = \frac{1}{\sqrt{\pm G(w)}}f_{v}(w)$. The hyperbolic angle $\varphi(w)$
satifies
\[\cosh^2 \varphi(w) = \lpr{\cos \theta \; \hat{a} + \sin \theta \;
    \hat{b}}{\cos \theta \; a + \sin \theta \; b} = \frac{1}{1 - u^{2} -
    v^{2}},\] and we note that $\varphi(w) \rightarrow \infty$ when
$\vert w \vert \rightarrow 1$.

On the other hand, we have that a basis for the normal bundle with
positive induced orientation is given by
\[N_{1}(w) = \frac{1}{\sqrt{\pm E(w)}}(u,-v,1,0)\quad\mbox{and}\quad
  N_{2}(w) = \frac{1}{\sqrt{\pm G(w)}}(v,u,0,1),\] showing that the
normal plane at each point is an positive plane isoclinic to $\pi_{34}$.

The second quadratic form of $S_{\pm}$ are given by
$B_{ij} = B_{ij}^1N_{1} + B_{ij}^2N_{2}$ where
\[B_{ij}^{1} = \lpr{D_{ij}f}{N_{1}} = 
  \frac{1}{\sqrt{\pm E}}
  \begin{bmatrix} 1 & 0 \\ 0 & -1 \end{bmatrix}
  \quad\mbox{and}\quad 
  B_{ij}^{2} = \lpr{D_{ij}f}{N_{2}} = 
  \frac{1}{\sqrt{\pm G}}
  \begin{bmatrix} 0 & 1 \\ 1 & 0 \end{bmatrix}.\]
Therefore, the Gauss curvatures are: 
\[K_{S_{-}}(w) = K_{S_{+}}(w) = \frac{\det B_{ij}^{1} + \det
    B_{ij}^{2}}{\det g_{ij}} = -\frac{2}{(1 - u^{2} - v^{2})^{3}}.\]

\subsection{On $\complex P^{3}$}
Let us extend the symmetric bilinear form $\langle ~,~\rangle$ of index
2 in $\real_{2}^{4}$ to the following symmetric bilinear form in the
$4$-dimensional complex vector space
$\complex^{4}\equiv \mathbb{R}^4\oplus i\mathbb{R}^4$:
\begin{equation}
  \ll x_1 + i y_1,x_2 + i y_2 \gg \, = \left(\lpr{x_1}{x_2} -
    \lpr{y_1}{y_2}\right) + i\left(\lpr{x_1}{y_2} +
    \lpr{y_1}{x_2}\right).
\end{equation}

As usual, we denote complex projective space of complex dimension 3 by
$\mathbb{C}P^3$, which corresponds to the space of all the complex lines
through the origin of $\mathbb{C}^4$.

Now, given any positive plane $V^2=\Span\{v,w\}\subset\semi$, where
$\{v,w\}$ is a $\lambda$-isothermic basis, and any complex number
$\mu \neq 0$, we have that
$\lbrace\mathfrak{Re}(\mu \,z),\mathfrak{Im}(\mu\,z)\rbrace$, with
$z = v + i w$, is a $\vert \mu \vert \lambda$-isothermic basis for the
plane $V^2$. Therefore, we can identify the plane $V^2$ with the point
$\left[ v + i w\right]\in\mathbb{C}P^{3}$, and define the Grassmannian
of the positive planes as a quadric in $\mathbb{C}P^{3}$ (See \cite{HO}
for details). Then we define the Grassmannian of the positive planes in
$\mathbb{R}^4_2$ as the complex subquadric of $\mathbb{C}P^3$:
\begin{equation}
  Q_{pos}^{2} = \left\{ \left[ z\right] \in \mathbb{C}P^{3} \colon  \ll
    z,z \gg = 0 \,\,\mbox{and}\,\, \ll z,\bar{z} \gg > 0 \right\}.
\end{equation}
In same way, the Grassmannian of the negative planes in $\mathbb{R}^4_2$
is the complex subquadric of $\mathbb{C}P^3$:
\begin{equation}
Q_{neg}^2 = \left\lbrace \left[ z\right] \in \mathbb{C}P^{3} \colon  \ll z,z
  \gg = 0 \,\,\mbox{and}\,\, \ll z,\bar{z} \gg < 0 \right\rbrace.
\end{equation}
\begin{remark}
  In $\semi$ there is pair of null vectors $x$ and $y$ such that
  $\langle x,y\rangle=0$, but they are linearly independent. More
  generally, the metrics index is the maximum dimension of a null
  subspace. For example, $x=(1,0,1,0)$ and $y=(0,1,0,1)$. In this case,
  writing $z=x+iy$ we have
  \[\ll z,z \gg = \ll z,\bar{z} \gg = 0.\]
\end{remark}

This remark shows that
\[Q^2_{null} =\left\lbrace [z]\in \mathbb{C}P^3 \colon \ll z,z\gg \,= \,
    \ll z,\bar{z} \gg \, = 0 \right\rbrace\] is not empty.
Therefore,
\[Q^2 = \left\lbrace [z]\in \mathbb{C}P^3\colon \, \ll z,z \gg\, =
    0\right\rbrace\] is the disjoint union of the subquadrics
$Q^2_{pos}$, $Q^2_{neg}$ and $Q^2_{null}$.





\begin{theor}
  Let $\overline{\complex} = \complex \cup \{\infty\}$ be the Riemann
  sphere. The map
  \begin{equation}
    \Phi([z]) = \left(\frac{z^{1} + i z^{2}}{z^{3} - i z^{4}},
      \frac{z^{1} - i z^{2}}{z^{3} - i z^{4}} \right) 
  \end{equation}
  is a homeomorphism from $Q^{2}$ onto
  $\overline{\complex} \times \overline{\complex}.$
\end{theor}

\begin{proof}
  We will provide a homogeneous coordinate system on $Q^2$. To this end,
  we see that $\left[ z\right]\in Q^2$ if and only if $-(z^{1})^{2} -
  (z^{2})^{2} + (z^{3})^{2} + (z^{4})^{2} = 0$,
  that is  
  \[(z^{1} + i z^{2})(z^{1} - i z^{2}) = (z^{3} + i z^{4})(z^{3} - i
    z^{4}).\] Suppose, for a moment, that $z^{3} - i z^{4} \neq 0$. Then
  we obtain
  \[\frac{z^{1} + i z^{2}}{z^{3} - i z^{4}} \; \frac{z^{1} - i
      z^{2}}{z^{3} - i z^{4}} = \frac{z^{3} + i z^{4}}{z^{3} - i
      z^{4}}.\] Defining the complex numbers
  $x = \frac{z^{1} + i z^{2}}{z^{3} - i z^{4}}$ and
  $y = \frac{z^{1} - i z^{2}}{z^{3} - i z^{4}}$, we get
  \begin{equation}\label{eq}
    z = \mu(x + y, -i(x - y), 1 + xy, i(1 - xy)),\mbox{ with } \;
    \; \mu = \frac{z^{3} - i z^{4}}{2}.
  \end{equation}
  If $z^{3} = \pm i z^{4} \neq 0$, then $z^{1} = \pm i z^{2}$. Hence, we
  can write
  \[z = \eta(x,\pm ix,1,\pm i), \mbox{ for $x = \frac{z^{1}}{z^{3}}$
      \mbox{ and } $\eta = z^{3}$.}\]
  
  Finally, if $z^{3} = i z^{4} = 0$, then $z = z^{1}(1, \pm i,0,0)$ is a
  complex representation of the plane $\pi_{12}\in Q_{neg}^{2}$.
  
  On the other hand, since $[z(x,y)] = [xy\,z(1/x,1/y)]$ it follows the:
  
  \begin{claim}
    If $[z(x,y)] = [x + y, i(x - y), 1 + xy, i(1 - xy)]$, then
    \begin{align*}
      \lim_{y \rightarrow \infty} [z(x,y)] &= [1,-i,x,-ix],
      \quad
        \lim_{x \rightarrow \infty} [z(x,y)] =
        [1,i,y,-iy]\\
      \mbox{and}
      & \lim_{(x,y) \rightarrow
        (\infty,\infty)}[z(x,y)] = [0,0,1,-i].
    \end{align*}
  \end{claim}
  
  Finally, we note that
  $\ll z,\overline{z}\gg \,= 2\mu \overline{\mu}(1 - x\overline{y})(1 -
  \overline{x}y) = 0$ if, and only if $[z] \in Q_{null}^{2}$, that
  corresponds to $ds^{2}|_{[z]} = 0$, the induced metric of these null
  planes.
\end{proof}


\begin{dfn}Let $M\subseteq \mathbb{C}$ be an open and connected set and
  $f\colon M\to \semi$ be a conformal immersion.  We say that $f$ is a
  positive (respec. negative) spacelike immersion if it is a parametric
  surface such that
  \[f_{w} = \frac{\partial f}{\partial w} =
    \frac{1}{2}\left(\frac{\partial f}{\partial u} - i \frac{\partial
        f}{\partial v}\right)\] satisfies $[f_{w}] \in Q_{pos}^{2}$
  (respec. $Q_{neg}^{2}$), where $w=u+iv$ is a conformal parameter for
  $M$.
\end{dfn}

\begin{remark}
  When the complex $1$-form $\beta = f_{w} dw$ has no real periods or,
  if $M$ is a simply connected open subset of $\complex$, the
  integral
  \[2\mathfrak{Re} \int_{w_{0}}^{w} \beta = \int_{w_{0}}^{w}
    2\mathfrak{Re}(\beta) = \int_{(u_{0},v_{0})}^{(u,v)} f_{u} du +
    f_{v} dv.\] is path independent. On the other hand, each real valued
  exact $1$-form $d \phi = \phi_{u} du + \phi_{v} dv$ can be written, in
  complex variables, as
  $d \phi = \phi_{w} dw + \phi_{\overline{w}} d \overline{w} = 2
  \mathfrak{Re} (\phi_{w} d w)$.
\end{remark}

As a consequence we have the following integral representation for
conformal immersions:

\begin{prop}[Weierstrass integral formula]\label{inte}
  Let $f\colon M\to \semi$ be a positive (respec. negative) spacelike
  conformal immersion. For any $w\in M$, we have that
  \begin{equation}\label{Wei}
    f(w) = f(w_{0}) + 2 \mathfrak{Re} \int_{w_{0}}^{w} \mu(\xi) W(x(\xi),y(\xi)) d \xi,
  \end{equation} 
  for some $[W(x(w),y(w))] \in Q_{pos}^{2}$
  $\left( [W(x(w),y(w))] \in Q_{neg}^{2}\right)$. The reciprocal is also
  true.
\end{prop}

\begin{dfn}Let $f\colon M\to \semi$ be a parametric surface. We say that
  $f$ is a minimal surface if its mean curvature vector $H_{f}$ vanishes
  identically.
\end{dfn}

\begin{lemma} Suppose that $f\colon M\to \semi$ is given by
  (\ref{Wei}). Then, $f$ is a minimal surface if, and only if
  $\mu(w),x(w)$ and $y(w)$ are holomorphic functions from $M$ into
  $\mathbb{C}$.
\end{lemma}

\begin{proof}
  The induced metric of such immersions is
  $\pm ds^2(f)=\lambda^2 \delta_{ij}du^i du^j$, since $f$ is a conformal
  parametric surface with $[f_w]\in Q^2_{pos}$ (or $Q^2_{neg}$ if
  negative). The mean curvature vector is defined by the
  Laplace-Beltrami equation (See \cite{HO} for details), hence
  \[\pm H_f =\dfrac{2}{\lambda^2}\Delta_M f = \dfrac{2}{\lambda^2}f_{w\overline{w}}=\dfrac{1}{2\lambda^2}(f_{uu} + f_{vv})=0.\]
  That is $H_f=0$ if, and only if
  $\left( \mu W(x,y)\right)_{\overline{w}}=0$. It follows from
  (\ref{eq}) that $\mu_{\overline{w}}=0$, $x_{\overline{w}}=0$ and
  $y_{\overline{w}}=0$.
\end{proof}

\subsection{The cross product in $\semi$} Given three vectors
$x,y,z \in \semi$, we define their \emph{cross product} as the unique
vector $\mathfrak{X}(x,y,z)$ such that, for each vector $v \in \semi$,
the following equation holds:
\begin{equation}
  \omega(x,y,z,v) = - \omega(v,x,y,z) = \lpr{\mathfrak{X}(x,y,z)}{v}.
\end{equation}

Defining the real numbers $\Delta_{ijk}$, for $i < j < k$ and
$i,j,k = 1,2,3,4$, by
\[\Delta_{ijk} = \left\vert
    \begin{matrix}
      x^{i} & x^{j} & x^{k} \\ y^{i} & y^{j} & y^{k} \\ z^{i} & z^{j} & z^{k}
    \end{matrix}\right\vert,\] 
we obtain the formal determinant defined by Laplace expansion on the
first line
\[\mathfrak{X}(x,y,z) = - \left\vert
    \begin{matrix}
      -e_{1} & -e_{2} & e_{3} & e_{4} \\ x^{1} & x^{2} & x^{3} & x^{4} \\ 
      y^{1} & y^{2} & y^{3} & y^{4} \\ z^{1} & z^{2} & z^{3} & z^{4}
    \end{matrix}\right\vert.\]

In coordinates, it takes the form
\[\mathfrak{X}(x,y,z) =(\Delta_{234}, -\Delta_{134}, -\Delta_{124}, \Delta_{123}).\]
Moreover, we have that
\[\omega(x,y,z,\mathfrak{X}(x,y,z)) = -
  \omega(\mathfrak{X}(x,y,z),x,y,z) = \sum_{i<j<k} (\Delta_{ijk})^{2}
  \geq 0.\]

Let $f\colon M\to \semi$ be a (positive or negative) spacelike conformal
immersion with normal bundle given by an orthonormal basis
$\left\lbrace A,B\right\rbrace$. From
$f_{w} = \frac{1}{2}(f_{u} - i f_{v})$, we have
\[ \mathfrak{X}(f_{u},A,B) =
  f_{v}\quad\mbox{and}\quad\mathfrak{X}(f_{v},A,B) = -f_{u},\] assuming
the positive orientation of the frame $\{A,B,f_u,f_{v}\}$. Hence,
\[\mathfrak{X}(f_{w},A,B) = \frac{1}{2}(\mathfrak{X}(f_{u},A,B) - i\mathfrak{X}(f_{v},A,B)) = 
  \frac{1}{2}(f_{v} + i f_{u})= if_{w}.\]

\subsection{A Bernstein-type problem}
The example \ref{exa} suggests the construction of spacelike parametric
surfaces in $\semi$ that are graphs of holomorphic functions
$Z(w)=\phi(w) + i\psi(w)$, defined in all complex plane. In this case,
the surface is given by
\begin{equation}
  f(w)=\left( u,v,\phi(u,v),\psi(u,v)\right).
\end{equation}
The induced metric for those positive (negative) spacelike conformal
surfaces is such that $\lambda^2(f)=-1+\vert Z'(w)\vert^2$. By the small
Picard Theorem of complex analysis, to obtain $\lambda^2>0$ in all
$\mathbb{C}$, the holomorphic function $Z'(w)$ need omits more than two
points, thus it is a constant function. Note that the mean curvature
vector $H_{f}$ vanishes, once $f_{w\overline{w}}(w)=0$ for each
$w\in \mathbb{C}$.

\begin{prop}Let $f\colon \mathbb{C}\to \semi$ to be the graph of a
  holomorphic function $Z(w)$. If it is a (positive or negative)
  spacelike parametric surface in $\semi$, then $Z(w)=aw+b$ for some
  constants $a,b\in \mathbb{C}$.
\end{prop}

Now, we will give an example of a non-trivial conformal parametric
surface in $\semi$, which has mean curvature vector $H \equiv 0$, defined
in all of the complex plane and it is free of singularities.

\begin{example}
  Let $x(w) = w$ and $y(w) = iw$, $w\in \mathbb{C}$ in the formula
  (\ref{Wei}). Then, we obtain
  \[f(w) = 2 \mathfrak{Re} \left(\frac{w^{2} - iw^{2}}{2}, i\frac{w^{2}
        + iw^{2}}{2}, \frac{3w - iw^{3}}{3}, i \frac{3w +
        iw^{3}}{3}\right)\] and
  $f_{w} = (w - iw, i(w + i w), 1 -iw^{2}, i(1 + iw^{2}))$. A direct
  computation shows that $\ll f_{w},f_{w} \gg \, = 0$,
  $\ll f_{w},{\overline{f_{w}}} \gg \, = 1 + \vert w \vert^{4} \geq 1$
  and $f_{w \overline{w}} = 0$. Therefore, $[f_{w}] \in Q_{pos}^{2}$ and
  $H_{f} = 0$.
\end{example}

\subsection{The normal bundle}

We have the following algebraic results:

\begin{lemma}
  Let $[v] = [v_1 + iv_2] \in Q_{neg}^{2}$ be and
  $[u] = [u_1 + i u_2] \in Q_{pos}^{2}$. The set $\{v_1,v_2,u_1,u_2\}$
  is an orthogonal basis of $\semi$ if, and only if
  $\ll v,u \gg \,= \,\ll v,\overline{u} \gg=0.$
\end{lemma}

\begin{proof}
  From
  \[\ll v,u\gg \, = (\lpr{v_1}{u_1} - \lpr{v_2}{u_2}) + i(\lpr{v_1}{u_2} + \lpr{v_2}{u_1}) = 0\]  and
  \[\ll v,\overline{u}\gg \, = (\lpr{v_1}{u_1} + \lpr{v_2}{u_2}) +
    i(\lpr{v_1}{u_2} - \lpr{v_2}{u_1}) = 0\] 
  it follows that
  \[\lpr{v_1}{u_1} = 0,\quad\lpr{v_1}{u_2} = 0,\quad\lpr{v_2}{u_1} =
    0\quad\mbox{and}\quad \lpr{v_2}{u_2}=0.\] Therefore, each vector of
  $[v]$ is orthogonal to any vector of $[u]$. In symbols,
  $[v] \perp [u]$.
\end{proof}

\begin{lemma}
  The planes $[u],[v]\in Q^2$, given by
\begin{align*}
  [u] &= [x + y, -i(x - y), 1 + xy, i(1 - xy)]\mbox{ and}\\
  [v] &= [1 + ab, i(1 - ab), a + b, -i(a - b)],
\end{align*}
are mutually orthogonal if, and only if either
\[a = x \mbox{ and } b = \overline{y}\quad\mbox{ or }\quad a =
  \dfrac{1}{\overline{x}}\mbox{ and } b = \dfrac{1}{y}.\]
\end{lemma}

\begin{proof}
  It follows from
  \begin{align*}
    \ll u,v\gg \,
    &= -(x + y + abx + aby) - (x - y - abx + aby)  \\ 
    &\qquad + (a + b + axy + bxy) + (a - b - axy + bxy)\\
    &= -2(x - a)(1 - by),
  \end{align*}
  and
  \begin{align*}
    \ll u,\overline{v}\gg \,
    &= -(x + y + \bar{a} \bar{b}x + \bar{a}\bar{b}y) + (x - y
      -\bar{a}\bar{b}x + \bar{a}\bar{b}y)  \\
    &\qquad+ (\bar{a} + \bar{b} + \bar{a}xy + \bar{b}xy) - (\bar{a} -
      \bar{b} - \bar{a}xy + \bar{b}xy)\\
    &=  - 2(y - \bar{b})(1 - \bar{a}x).
  \end{align*}
\end{proof}

Let $[w]=[ x + y, -i(x - y), 1 + xy, i(1 - xy)]\in Q$. Considering the
operator $\tilde{J} \colon Q^{2} \to Q^{2}$ defined by
\[\tilde{J}\left([w]\right) = \left[ 1 + x \overline{y}, i(1 - x\overline{y}), x + \overline{y},
    -i(x - \overline{y}) \right]\] we note that
$\tilde{J}([z]) \perp [z]$ and have the

\begin{prop}\label{jot}
  The normal operator $\tilde{J}$ is a smooth injective map from $Q^{2}$
  onto $Q^{2}$ satisfying:
  \[\tilde{J}(Q_{neg}^{2}) = Q_{pos}^{2}, \quad\tilde{J}(Q_{pos}^{2}) =
    Q_{neg}^{2}\quad \mbox{and}\quad\tilde{J}(Q_{null}^{2}) =
    Q_{null}^{2}.\] Furthermore, if we take
  $[W(x,y)] = [x + y, -i(x - y), 1 + xy, i(1 - xy)]$ then
  \[\tilde{J}([W(x,y)]) = [W(x,1/\overline{y})].\] 
  As expected, it is an idempotent operator on $Q^{2}$, that is,
  $\tilde{J} \circ \tilde{J} = Id$.
\end{prop}

\section{A Cauchy problem for isoclinic spacelike surfaces in
  $\semi$}\label{sec:cauchyp}

From now on, we will assume that $I=(-r,r)$ is a non-empty interval in
the real line $\mathbb{R}\subset \mathbb{C}$ and $M\subset \mathbb{C}$
is a simply connected and connected open subset such that $I\subset M$.

We would like to recall the following facts of complex analysis:
\begin{enumerate}
  \item Let $x, y\colon M\to\complex$ be two holomorphic functions such
  that $x(y)=y(y)$, forall $t\in I$. Since $I$ has accumulation points,
  we have that $x(w)=y(w)$, forall $w\in M$.
  \item Given any real analytic function $x\colon I\to\real$, the line
  integral \[x(w) = x(0) + \int_0^w x'(t)dt\] is a holomorphic function
  $x\colon M\to \mathbb{C}$, which is the unique holomorphic extension
  of the real valued function $x(t)$.
\end{enumerate}

Now, if $c(t)=(c_1(t),c_2(t),c_3(t),c_4(t))$ is a real analytic curve
from $I$ into $\semi$, the holomorphic function
$C(w)=(c_1(w),c_2(w),c_3(w),c_4(w))$ from $M$ into $\mathbb{C}^4$, where
$c_i(w)$ is the holomorphic extension of the function $c_i(t)$, is the
unique holomorphic extension of the curve $c(t)$.

We propose the following Cauchy problem, which consists of constructing
an isoclinic surface under certain conditions.

\begin{question}\label{pro}
  Given a positive (respec. negative) real analytic curve
  $c\colon I\to \semi$, thus $\langle c'(t),c'(t)\rangle > 0$
  (respec. $<0$), construct a conformal immersion
  $f\colon M\to \semi$ satisfying the following conditions:
  \begin{enumerate}
    \item $f(t,0)=c(t)$ for each $t\in I$ (extension condition),
    \item $H_f(w)=0$ for each $w\in M$ (minimality condition),
    \item $\left[ f_w(w)\right]\in Q^2_{pos}$ (respec.
    $\left[ f_w(w)\right]\in Q^2_{neg}$) and $f(M)$ is isoclinic to
    $\pi_{34}$ (respec. $pi_{12}$), for each $w\in M$.
  \end{enumerate}
\end{question}

From the first condition, considering a conformal parameter $w=u+iv$ for
$M$, a solution will satisfy $f_u(t,0)=c'(t)$ and hence it should have
the vector field $f_v(t,0)$ to be pointwise orthogonal to $c'(t)$, with
$\langle f_v(t,0),f_v(t,0)\rangle =\langle c'(t),c'(t)\rangle$.

Using the linear transformation given by (\ref{trans}), we define an
isoclinic distribution, along our curve $c(t)$, by
\[\D(t)=\left[ c'(t) - iL(c'(t))\right],\quad t\in I.\] 
With the usual extension of the operator $L$ to the complex vector space
$\mathbb{C}^4$, also denoted by $L$, we have the following map
$f\colon M\to \semi$:
\begin{equation*}
  f(w)=c(0) + \mathfrak{Re} \int_0^w (C'(\xi) - iL(C'(\xi)))d\xi,
\end{equation*} 
where $C(w)$ is the unique holomorphic extension of $c(t)$.

Since $C(t,0)=c(t)$ we have that $f(t,0)=c(t)$. It follows from
$f_w(w)=\dfrac{1}{2}(C'(w)-iL(C'(w)))$ that $f_{w\overline{w}}=0$. To
see that this map satisfies the conditions in our problem, we need of
the following algebraic lemma:

\begin{lemma}\label{idn}
  Let $z=x+iy\in \mathbb{C}^4$ be with $x$ and $y$ positive
  (respec. negative) vectors in $\semi$. Then
  \[w=z + iL(z) = x - L(y) + i(L(x) + y)\] satisfies  $\ll w,w\gg=0$ and
  $\ll w,\overline{w}\gg > 0$ (respec. $\ll w,\overline{w}\gg < 0$).
\end{lemma}

\begin{proof}
  Since
  \begin{align*}
    \langle L(y),L(y)\rangle =\langle y,y\rangle,&\quad \langle
    L(x),L(x)\rangle=\langle x,x\rangle\\
    \mbox{and}\quad\langle x,L(y)\rangle&= -\langle L(x),y\rangle,
  \end{align*}
  direct computations show that
  \[\ll w,w\gg=0\quad\mbox{and}\quad \ll w,\overline{w}\gg =2(\langle
    x,x\rangle + \langle y,y\rangle).\] Thus, being $x$ and $y$ positive
  (respec. negative) vectors in $\semi$, we have
  $\ll w,\overline{w}\gg>0$ (respec.  $\ll w,\overline{w}\gg<0$).
\end{proof}

\begin{remark}
  We see from the example~\ref{exa} that we may need to restrict the
  domain of the solutions of our problem, since we can obtain positive
  or negative isoclinic surfaces with metric singularities when we
  extend the given functions.
\end{remark}

We can establish the following result:

\begin{theor}
  Let $c\colon I\to \semi$ be a positive (respec. negative) real analytic
  curve. There exists a connected and simply connected open subset
  $M\subset \mathbb{C}$ with $I\subset M$ and $f\colon M\to\semi$ give
  by
  \begin{equation}\label{fun}
    f(w)=c(0) + \mathfrak{Re} \int_0^w (C'(\xi) - iL(C'(\xi)))d\xi,
  \end{equation} 
  such that its image is a positive (respec. negative), isoclinic and
  minimal conformal immersion, without metric singularities, such that
  $f(t,0)=c(t)$. Moreover, this is the unique solution of the Cauchy
  problem\,\ref{pro}.
\end{theor}

\begin{proof}
  The prevoius Lemma ensures that (\ref{fun}) is a conformal immersion
  satisfying the conditions (2) and (3) of our problem. Now, if
  $g\colon M\to \semi$ is another solution, then $g$ satisfies the
  condition (1), thus $f$ and $g$ coincide on the interval $I$. Hence,
  they must coincide on the domain $M$.
\end{proof}

\begin{example}
  Let $\Psi(z)$ and $\Phi(z)$ be two holomorphic functions from $M$ into $\mathbb{C}$. The map 
  \[f(w)=\left( \dfrac{\Psi(w)+\overline{\Psi(w)}}{2},\dfrac{\Psi(w) -
        \overline{\Psi(w)}}{2i},\dfrac{\Phi(w) +
        \overline{\Phi(w)}}{2},\dfrac{\Phi(w)
        -\overline{\Phi(w)}}{2i}\right),\] parametrizes an isoclinic
  minimal surface. When $\vert\Psi'\vert < \vert\Phi'\vert$ we have a
  positive parametric surface and when
  $\vert\Psi'\vert > \vert\Phi'\vert$ it is a negative one. Furthermore,
  the set of metric singularities is
  \[Sing(f)=\left\lbrace w\in M \colon \vert \Psi'\vert = \vert
      \Phi'\vert\right\rbrace \subset S^4_2(0).\]
\end{example}
In fact, straightforward calculations shows that
$\langle f_w,f_w\rangle=0$ and the minimality of the surface. Also,
since $2f_w=\left( \Psi',-i\Psi',\Phi',-i\Phi'\right)$ it follows that
$L(f_v)=-f_u$ and $L(f_u)=f_v$. Thus, by Proposition\,\ref{apl}, the
tangent plane at each point of the surface is isoclinic to $\pi_{34}$
(respec. $\pi_{12}$).

\begin{theor}
  Let $f\colon M\to \semi$ be an isoclinic conformal immersion. Then,
  there exists two holomorphic functions $\Psi(z)$ and $\Phi(z)$ from
  $M$ into $\mathbb{C}$ such that
  \begin{equation}
    f(w)=\left(\mathfrak{Re}\Psi(w),\mathfrak{Im}\Psi(w),\mathfrak{Re}\Phi(w),\mathfrak{Im}\Phi(w)\right).
  \end{equation}
\end{theor}

\begin{proof}
  Since the positive (negative) plane $\Span\{\hat{f}_u,\hat{f}_v\}$ is
  given by an isothermic basis, we must have, from Proposition
  \ref{apl}, that $f_v=L(f_u)$. Writing $f=(f^1,f^2,f^3,f^4)$ this
  implies, $f^1_u=f^2_v$ and $f^1_v=-f^2_u$, that is, $\Psi=f^1+if^2$ is
  a holomorphic function on $M$. Analogously, we have that
  $\Phi=f^3 + if^4$ is holomorphic and the result follows.
\end{proof}

In particular, for $\Psi(w)=w$ and $\Phi(w)=w^2$ we recover example \ref{exa}.

Recall that, for an abstract Riemannian surface endowed with metric
tensor $\pm ds^2= \lambda^2(u,v)(du^2 + dv^2)$, the Gauss curvature is
given by
\[K(w)=-\dfrac{1}{\lambda^2} \left( \left(
      \dfrac{\lambda_u}{\lambda}\right)_u + \left(
      \dfrac{\lambda_v}{\lambda}\right)_v \right).\]
Hence we have the:

\begin{corol}
  Any isoclinic conformal immersion, $f\colon \mathbb{C}\to \semi$
  (without metric singularities), is flat: $K(w)=0$ for all
  $w\in \mathbb{C}$.
\end{corol}

\begin{proof}
  If $0\leq \vert \Psi'\vert < \vert \Phi'\vert$, then by Liouville
  Theorem there exists $a\in \mathbb{C}$, with $\vert a\vert < 1$, such
  that $\vert \Psi'/\Phi'\vert = \vert a\vert$. Hence,
  $\epsilon ds^2=\vert \Phi'\vert^2(\vert a\vert^2 - 1)(du^2 +
  dv^2)$. Since $\Phi'$ is holomorphic and then
  $\Delta \ln\vert \Phi'\vert =0$, the surface is flat.
\end{proof}

\section{The Bj\"orling Problem}\label{sec:bjorlingp}

In this section we will deal with the following problem, known as the
Bj\"orling Problem.

\begin{question}\label{prob:bjorling}
  Let $I=\left(-r,r\right)$ be an open interval and, for any $t\in I$,
  $\left( c(t),A(t),B(t)\right)$ a given traid, where $c(t)$ is a
  positive (respec. negative) real analytic curve in $\semi$ and $\left\lbrace
    A(t),B(t)\right\rbrace$ is a family of orthonormal sets such that
  $\left[ A(t)+iB(t)\right]\in Q^2_{neg}$ (respec. $\left[
    A(t)+iB(t)\right]\in Q^2_{pos}$ and defines a real analytic
  distribution along the curve $c(t)$ pointwise orthogonal to the
  tangent vector $T'(t)=c'(t)/\Vert c'(t)\Vert$.
  
  Our aim is to construct a positive (resepc. negative) spacelike
  conformal immersion $f\colon M\to \semi$, with $I\subset M$,
  satisfying the following conditions:
  \begin{enumerate}
    \item $f(t,0)=c(t)$ for all $t\in I$,
    \item The normal bundle of the surface, along the curve $c(t)$,
    satisfies $N_{c(t)}f(M)=\left[ A(t)+iB(t)\right]$ for all $t\in I$,
    \item $H_f(w)=0$ for all $w\in M$.
  \end{enumerate}
\end{question}




To set and solve this problem we need to assume some good conditions for
the curve $c(t)$. We need, pointwisely, that the orthogonal complement
to the tangent vector $c'(t)$ is isomorphic to $\mathbb{R}^3_2$, when
the curve is positive, or isomorphic to $\mathbb{R}^3_1$, when the curve
is negative.

\begin{example}
  The curve
  \[c(t) = \left( \cos t,\sin t,2\cos \left( t/\sqrt{2}\right),2\sin
      \left( t/\sqrt{2}\right) \right)\] in $\semi$ is a positive curve
  such that $\lpr{c'}{c'} = 1$, $\lpr{c''}{c''} = 0$ and
  $\lpr{c'''}{c'} = \lpr{c'''}{c''} = 0$. Moreover, for $k=3,4,...$ it
  follows that $\langle c^{(k)},c^{(k)}\rangle < 0$. There is no
  positive vector orthogonal to $c'$, $c''$ and $c'''$, hence no
  positive spacelike surface containing the curve.
\end{example}

\begin{dfn}
  Let $c \colon I \to \real_{2}^{4}$ be a positive regular curve. We say
  that $c$ is a good curve, if there exists a frame
  \[\left\lbrace T(s),N(s),B(s),R(s)\right\rbrace_{s \in I}\] 
  which is an orthonormal positive basis of the space $\real_{2}^{4}$,
  satisfying the following conditions:
  \begin{enumerate}[label=(\roman*)]
    \item $c'(s) = v(s) T(s)$ and $\lpr{c'(s)}{c'(s)} = (v(s))^2 > 0$, 
    \item $N(s) \in \Span\{c'(s),c''(s)\}$,
    $\vert\langle N(s),N(s)\rangle\vert=1$,
    \item $B(s) \in \Span\{c'(s),c''(s),c'''(s)\}$ and $\vert\langle
    B(s),B(s)\rangle\vert=1$ and
    \item $R(s) = \mathfrak{X}(T(s),N(s),B(s))$.
  \end{enumerate}


\end{dfn}

In the same way, a definition for good negative curves can be given.

To solve the Problem \ref{prob:bjorling}, we will use the so called
Schwarz integral equation.

\begin{prop}(Schwarz integral equation) Let $I=(-r,r)$ be an open
  interval and $(c(t),A(t),B(t))_{t\in I}$ the triad as in Problem
  \ref{prob:bjorling}. If $M\subset\complex$ is an open, connected and
  simply connected set containg $I$, let the $f:M\to\complex$ to be the
  map given by
  \begin{equation}
    f(w) = c(0) + 
    \mathfrak{Re} \int_{0}^{w} \left(C'(\xi) - i
      \mathfrak{X}(C'(\xi),A(\xi),B(\xi))\right)d\xi,
  \end{equation}
  where $C(w)$, $A(w)$ and $B(w)$ are the (unique) analytic extensions
  of the curve $c(t)$ and the vector fields $A(t)$ and $B(t)$,
  respectively.

  Then it follows from (\ref{Wei}) that $f$ define a minimal conformal
  immersion such that $f(t,0)=c(t)$, such that its normal bundle along
  $c$ is $N_{c(t)}f(M)=\left[ A(t) + iB(t)\right]$.
\end{prop}

Setting $d'(w)=\mathfrak{X}(C'(w),A(w),B(w))$, from
$f_w=\dfrac{1}{2}(C'(w) - id'(w))$, it follows that
$f_w(t,0)=\dfrac{1}{2}(c'(t)+id'(t)$ and, since $c'$ and $d'$ are vector
fields in $\mathbb{R}^4_2$ along the curve, we obtain:
\[\dfrac{\partial f}{\partial
    u}(t,0)=c'(t)\quad\mbox{and}\quad\dfrac{\partial f}{\partial
    v}(t,0)=d'(t).\]

The following result will be useful:

\begin{lemma}
  Suppose that, in the above setting, the elements of the triad
  $(c(t), A(t), B(t))_{t\in I}$ are such that $c$ is a good curve and
  the frame
  $\left\lbrace A(t),B(t)\right\rbrace\subset
  \Span\{N(t),B(t),R(t)\}$. If we set the complex vector field
  \begin{equation}\label{swa}
    d'(w)=\mathfrak{X}(C'(w),A(w),B(w)),
  \end{equation}
  then $\left[ C'(w) - id'(w)\right]\in Q^2_{pos}$.
\end{lemma}

\begin{proof}
  Since $d'(w)$ is the analytic extension of the real vector field
  $d'(t)=\mathfrak{X}(c'(t),A(t),B(t))$ and the cross product is a
  3-linear operator, let $\mathfrak{X}(C'(w),A(w),B(w))$ be its
  extension to $\mathbb{C}^4$, that restricted $(t,0)$ give us
  $d'(t)$. Hence the equation (\ref{swa}) holds. By Lemma~\ref{idn} we
  also have $\left[ C'(w)-id'(w)\right] \subset Q^2_{pos}$.
\end{proof}

\begin{lemma}[Existence]
  Given a triad $\left( c(t),A(t),B(t)\right)_{t\in I}$ on the interval
  $I=\left(-r,r\right)$, satisfying the conditions in Problem
  \ref{prob:bjorling} such that $\left[ A(t)+iB(t)\right]\in Q^2_{neg}$
  ($\left[ A(t)+iB(t)\right]\in Q^2_{pos}$, if the curve in negative)
  Then the map $f:M\to\semi$, defined on a certain complex domain $M$
  containing $I$, given by
  \begin{equation}\label{exis}
    f(w) = c(0) + 
    \mathfrak{Re} \int_{0}^{w} \left(C'(\xi) - i \mathfrak{X}(C'(\xi),A(\xi),B(\xi))\right)d\xi, 
  \end{equation}
  is a conformal immersion such that $H_f =0$
  and
  \[f(t,0)=c(t)\quad\mbox{and}\quad\dfrac{\partial f}{\partial
      w}(t,0)=c'(t) - i\mathfrak{X}(c'(t),A(t),B(t)).\]
\end{lemma}

\begin{proof}
  Taking $d'(w)$ given by (\ref{swa}) and using the operator given in
  Proposition~\ref{jot}, we have that
  \[\tilde{J}(\left[ A(w)+iB(w)\right])=\left[ C'(w)-id'(w)\right]=\left[
      z\right]\] and the holomorphic functions from $M$ into
  $\mathbb{C}$ given by (\ref{Wei}) in Theorem~\ref{inte}: $\mu(w)$,
  $x(w)$ and $y(w)$. These functions define the unique holomorphic
  extension
  \begin{align*}
    f'(w)
    &=C'(w) -id'(w)\\
    &=C'(w)-i\mathfrak{X}(C'(w),A(w),B(w))\\
    &=2\mu(w)W(x(w),y(w)).
  \end{align*}
  Hence, (\ref{exis}) is a solution of our problem.
\end{proof}

\begin{lemma}[Uniqueness]
  Assume that $(M',g(z))$ and $(M,f(w))$ are two solutions of Problem
  \ref{prob:bjorling} such that $c(I) \subset f(M) \cap g(M')$, and
  $[g_{w}(w(t))] = [f_{z}(z(s))]$ for all $(s,t) \in I \times I'$ also
  satisfying $g(w(t)) = f(z(s)) = P(t) \in C(I)$. Then, the subset
  $f(M) \cap g(M') \supset c(I)$ is an open subset of the both surfaces
  $f(M)$ and $g(M')$.
\end{lemma}
\begin{proof}
  Define the real valued function $s = \rho(t)$ such that
  $w(t) = z(\rho(t))$, we have that $g(w(t)) = f(z(\rho(t)))$. Since $g$
  and $f$ are the real parts of holomorphic fields in $\complex^{4}$,
  the function $\rho$ is a real analytic function. Let $\eta(w)$ be a
  local holomorphic extension of the function $\rho(t)$, that is
  $\eta(t,0) = \rho(t)$.

  Setting $h(w) = f(z(w))$, we have that $h(U) \subset f(M)$ and
  $h_{w}(w) = z'(w)f_{z}(z(w))$. Since, by hypothesis,
  \[[g_{w}(w(t))] = [f_{z}(z(\rho(t)))] = [h_{w}(z(\rho(t)))]
    =[W(x(t),y(t)],\]along the curve $c(I)$ we have:
  \begin{align*}
    g_{w}(t,0) &= \mu(t) W(x(t),y(t))\\
               &= \chi(t)z'(\rho(t)) W(x(t),y(t)) = \beta(t) W(x(t),y(t))\\
               &= h_{w}(t,0). 
  \end{align*}

  Therefore, $\mu(t) = \chi(t)z'(\rho(t)) = \beta(t)$ and then we
  obtain, by unicity of the holomorphic extensions,
  \begin{align*}
    h(w) &= P + 2 \Re \int_{w_{0}}^{w} \beta(\xi) W(x(\xi),y(\xi)) d
           \xi\\
         &= P + 2 \Re \int_{w_{0}}^{w} \mu(\xi) W(x(\xi),y(\xi)) d \xi\\
         &= g(w)
  \end{align*}
  in a neighborhood of $c(I)$.
\end{proof}

Summarizing, we have the:

\begin{theor}
  Leti $I=(-r,r)$ be an open interval and assume given the analytic
  triad $(c(t),A(t),B(t))_{t\in I}$, composed by a positive
  (respec. negative) regular curve and a family of orthonormal sets such
  that $\Span\{A,B\} \subset \Span\{N,B,R\}$ and $[A + iB] \in Q^{2}$ is
  a negative (resepc. positive) distribution along this curve.

  The integral Schwarz formula
  \begin{equation*}
    f(w) = c(0) + 
    \Re \int_{w_{0}}^{w} \left(c_{w}(\xi) - i
      \mathfrak{X}(c_{w}(\xi),A(\xi),B(\xi)) \frac{}{}\right) d \xi,
  \end{equation*}
  where the complex triad $(c_{w}(w),A(w),B(w))$ is the (unique)
  holomorphic extension of the triad $(c'(t),A(t),B(t))$, is a maximal
  solution of the Problem~\ref{prob:bjorling}, without metric
  singularities.
\end{theor}

\section{Constructions and examples} 

\begin{example}
  As our first example consider the regular curve
  \[c(t) = \left( t,0,t^{2}/2,0\right),\quad t\in I=(-1,1).\] This is a
  negative curve with $c'(t)=(1,0,t,0)$. The two remaining elements of
  the triad are
  \[A(t) = \frac{1}{\sqrt{1 - t^{2}}}(t,0,1,0)\quad\mbox{and}\quad B(t)
    = -\frac{1}{\sqrt{1 - t^{2}}}(0,t,0,1),\,t\in I.\] A direct
  computation gives $d'(t)=\mathfrak{X}(c'(t),A(t),B(t))=(0,1,0,t)$. The
  holomorphic extension of $d'(t)$
  is \[d'(w)=\mathfrak{X}(c'(w),A(w),B(w))=\left( 0,1,0,w\right).\]
  Hence,
  \begin{align*}
    f(w)&=\mathfrak{Re}\int_{0}^{w} ((1,0,\xi,0) - i(0,1,0,\xi))d\xi\\
        &=\mathfrak{Re}\left( \left(w,0,w^{2}/2,0\right) - i\left(
          0,w,0,w^{2}/2\right) \right).
  \end{align*}
  Therefore, the solution of the Bj\"orling problem for the given data
  is
  \[f(u,v) = (u,v,(u^{2}- v^{2})/2,uv),\mbox{ with $w\in \mathbb{C}$ and
      $\vert w\vert< 1$},\] which is the surface of example~\ref{exa}.
\end{example}

\subsection{Periodic curves in the Lorentzian Torus of the null sphere}
  
Let $T_{1}^{2}$ be the Torus in the null sphere $S^3_2(0)\subset\semi$,
parametrically given by \[T(u,v) = (\cos u, \sin u, \cos v, \sin v),\]
for $(u,v) \in \real^{2}$. The induced metric on this surface is
$ds^{2}(T_{1}^{2}) = - du^{2} + dv^{2}$, so it is a compact Lorentz
surface in $\semi$ isometric to the Lorentz plane $\real_{1}^{2}$.

Now, for a real number $k > 0$, we take the parametric curve
$c(t) = T(t,kt)$. For this curve, $\lpr{c'}{c'} = -1 + k^{2}$ and
$\lpr{c''}{c''} = -1 + k^{4}$.  Therefore, if $k > 1$ we obtain a
positive parametric curve.

\begin{claim}\label{claim:cycles}
  For any integer $n\geq 2$, we have that
  \[c_{n}(t) = T(t,nt) = (\cos t, \sin t, \cos nt, \sin nt) = a(t) +
    b(nt),\] where $a(u) = T(u,0)$ and $b(v) = T(0,v)$, is a periodic
  positive good curve.
\end{claim}

Now, we take the parametric surface given by
\[X_{n}(r,\theta) = \left(r \cos \theta, r \sin \theta, r^{n} \cos n
    \theta, r^{n} \sin n \theta\right).\] Its induced metric is
\[ds^{2} = (-1 + n^{2} r^{2(n-1)}) dr^{2} + r^{2}(-1 + n^{2} r^{2(n-1)})
  d\theta^{2}.\] The parametric surface $X_{n}(r,\theta)$ is, for
$r > 1/n$ and $n > 1$, a positive surface for such that
$X(1,t) = c_{n}(t)$.

\begin{claim}
  The periodic positive parametric curve $c_{n}(t)$ traces the cycle of
  $T_{1}^{2}$: $c([0,2\pi]) = T_{1}^{2} \cap S$, where $S \subset \semi$
  is the graph of the holomorphic function $\phi(z) = z^{n}$, for
  $\vert z \vert > 1/n$.
\end{claim}

With the following unit and mutually orthogonal curves
\begin{align*}
  a(t) &= (\cos t,\sin t,0,0),\ x(t) = a'(t),\\
  b(t) &= (0,0,\cos t,\sin t)\mbox{ and } y(t) = b'(t)
\end{align*}
we have
\begin{align*}
  c'(t) &= x(t) + n y(nt),\quad c''(t) = - a(t) - n^{2}b(nt)\\
        &\mbox{ and } c'''(t) = -x(t) - n^{3} y(nt).
\end{align*}
Hence, the negative vector field $V(t) = n^{2} a(t) + b(nt)$ is
orthogonal to $c'$, $c''$ and $c'''$. Considering
$V'(t) = n^{2} x(t) + ny(nt)$, we obtain a negative plane
$\Span\{V'(t),V(t)\}$, that is orthogonal to the plane
$[c'(t),c''(t)])$.

By construction $V = \alpha \mathfrak{X}(c',c'',c''')$, for a some
$\alpha \in \real$. Therefore the Frenet frame for the curve $c(t)$ is
\begin{align*}
  T(t) = &c'(t),\quad N_{1}(t) = \eta c''(t),\quad N_{3}(t) = \xi V(t)\\
         &\mbox{and } N_{2}(t) = \mathfrak{X}(T(t),N_3(t),N_1(t)), 
\end{align*}
for some real numbers $\eta$ and $\xi$.

\begin{claim}
  The surface $S=graph(z^n)$ has
  \[X(r,\theta) = (r \cos \theta, r \sin \theta, r^{n} \cos n \theta,
    r^{n} \sin n \theta)\] as a parametrization, hence its normal plane
  at each point is given by $[A(r,\theta) + i B(r,\theta)]$, where
  \begin{align*}
    A(r,\theta) &= \frac{1}{\sqrt{-1 + n^{2} r^{2n-2}}} (n r^{n - 1}
                  \cos n \theta, n r^{n - 1} \sin n \theta, \cos \theta,
                  \sin \theta)\\
    B(r,\theta) &= \frac{1}{r \sqrt{-1 + n^{2} r^{2n-2}}} (-n r^{n} \sin
                  n \theta, n r^{n} \cos n \theta,-r \sin \theta,r \cos
                  \theta).
  \end{align*}
  Therefore, along the curve $r = 1$ and $\theta = t$ we obtain:
  \[A(t) = \frac{1}{\sqrt{n^{2} - 1}}(n a(nt) +
    b(t))\quad\mbox{and}\quad B(t) = \frac{1}{\sqrt{n^{2} - 1}}(n x(nt)
    + y(t)).\]

  For the triad $(c(t) = X(1,t), A(t), B(t))$ we obtain
  \begin{align*}
    c'(t) &= X_{\theta}(1,t) = x(t) + n y(nt)\quad\mbox{and}\\
    d'(t) &= X_{r}(1,t) = \mathfrak{X}(c',A,B) = -a(t) - n b(nt).
  \end{align*}
\end{claim}

\begin{example}
  For the triad $(c(t),A(t),B(t))$, where
  \begin{align*}
    c(t) &= a(t) + b(nt),\\
    A(t) &=\frac{1}{\sqrt{n^{2} - 1}}(n a(nt) + b(t))\quad\mbox{and }\\
    B(t) &= \frac{1}{\sqrt{n^{2} - 1}}(n x(nt) + y(t)),
  \end{align*}
  the unique solution of Problem~\ref{prob:bjorling} is given by
  \begin{align*}
    f(w) &= \mathfrak{Re}(e^{iw}, -ie^{iw}, e^{inw}, -ie^{inw})\\
         &= \mathfrak{Re}(z, -iz, z^{n}, -iz^{n}),\quad\mbox{if $z =
           e^{iw}$}.
  \end{align*}
\end{example}

In fact, note that:
\begin{align*}
  f(t,0) &= \mathfrak{Re}(e^{it}, -ie^{it}, e^{int}, -ie^{int}) =
           c(t)\mbox{ and}\\
  \frac{\partial f}{\partial w}(t,0)
         &= i(e^{iw}, -ie^{iw}, ne^{inw}, -ine^{inw})\vert_{(t,0)} = c'(t) - i d'(t).
\end{align*}

We need to check if this $d'(t)$ has the right signal. It suffices to
verify this at $t=0$:
\[(n^{2} - 1)\mathfrak{X}(c'(0),A(0),B(0)) =
  \left\vert
    \begin{matrix}
      e_{1} & e_{2} & -e_{3} & -e_{4} \\ 0 & 1 & 0 & n \\
      n & 0 & 1 & 0 \\ 0 & n & 0 & 1
    \end{matrix}
  \right\vert = (n^{2}-1)(e_{1} + ne_{3}).\] So we have
$d'(t) = -(a(t) + n b(nt))$, as expected.

\begin{remark}
  It is well know that an arc-length parametric regular and \emph{good}
  curve $c(t) = f(u(t),v(t))$ of a positive (respec. negative)
  parametric surface $(M,f)$ is a (pre-)geodesic line if, and only if,
  its normal vector field lies on the normal bundle $N_{c}f(M)$ along
  this curve. For asymptotic lines, its normal vector fields is positive
  (respec. negative) and lies on the tangent bundle along the curve.
\end{remark}

On Example \ref{exa}, the negative curve
\[\alpha(t) = \frac{1}{2}\left(\cos t, \sin t, \frac{\cos 2t}{2},
    \frac{\sin 2t}{2}\right)\] is such that $\lpr{c''(t)}{c''(t)} = 0$,
for all $t \in \real$. This curve corresponds to $u^{2} + v^{2} =
1/4$. Condintions for a curve $c(t)$ in the triad for the Björling
problem to be a geodesic or an asymptotic line are in the

\begin{theor}
  Let $c(t)$ be a positive arc-length parametric good curve with Frenet
  frame $\{T(t),N_{1}(t),N_{2}(t),N_{3}(t)\}.$ If the first normal
  $N_{1}(t)$ and $V(t) \in \{N_{2}(t),N_{3}\}$ are negative vector
  fields along this curve, then the solution $(M,f)$ of the
  Problem~\ref{prob:bjorling} for the triad $(c(t),N_{1}(t),V(t))$ has
  the curve $c(t) = f(t,0)$ as a geodesic line.

  On the other hand, if $N_{1}(t)$ is a positive vector field along the
  curve $c(t)$, then the solution $(M,f)$ of the
  Problem~\ref{prob:bjorling} for the triad $(c(t),N_{2}(t),V(t))$ has
  the curve $c(t) = f(t,0)$ as an asymptotic line.
\end{theor}

Considering the cycles given in Claim~\ref{claim:cycles}, that is
$c(t) = a(t) + b(nt)$, and the normal distribution
$\D(t)=[N_{2}(t),N_{3}(t)]$ we have

\begin{example}
The parametric surface given by 
$$f(w) = c(0) + \mathfrak{Re} \int_{0}^{w}\left(c'(\xi) - i \frac{1}{\sqrt{1 + n^{2}}} c''(\xi) \right) d\xi$$ 
is a solution of the Problem 3.3 such that the curve $c(t) = f(t,0)$ is a asymptotic line of this new surface $f(M)$.
\end{example} 

\subsection{A class of Helices}
For positive real numbers $0 < b < a < \sqrt{b}$, we take in $\semi$ the curve
$$\alpha(s) = \left(b \cos \frac{s}{\sqrt{a^{2} - b^{2}}},b \sin \frac{s}{\sqrt{a^{2} - b^{2}}}, 
\cos \frac{as}{\sqrt{a^{2} - b^{2}}}, \sin \frac{as}{\sqrt{a^{2} - b^{2}}} \right).$$ 

Since $\lpr{\alpha'}{\alpha'} = 1$ and
$\lpr{\alpha''}{\alpha''} = \dfrac{a^{4} - b^{2}}{(a^2 - b^2)^2} <0$,
this is a positive curve with negative normal vector field along
$\alpha(s)$. Let $t=\dfrac{s}{\sqrt{a^2-b^2}}$ to shorten notation. We
have
\begin{align*}
  T(s) &= \frac{1}{\sqrt{a^{2} - b^{2}}}\left(-b \sin t,b \cos
         t,-a \sin at,a \cos at\right),\\
  N(s) &= \frac{-1}{\sqrt{b^{2} - a^{4}}}\left(b \cos t,b \sin t,
         a^{2} \cos at, a^{2} \sin at\right),\\
  N'(s) &= \frac{-1}{\sqrt{b^{2} - a^{6}}}\left(-b \sin t, b \cos
          t, -a^{3} \sin at, a^{3} \cos
          at\right).
\end{align*}

Let us define the following set of unit and mutually orthogonal curves:
\begin{align*}
  x(t) &= (\cos t,\sin t,0,0),\  y(t) = x'(t), \\
  \eta(t) &= (0,0,\cos at,\sin at),\mbox{ and } \xi(t) = \frac{1}{a}\eta'(t).
\end{align*}

In this way we write
\[\alpha' = b y(t) + a \xi(t),\ \alpha'' = - b x(t) - a^{2}
  \eta(t)\mbox{ and }\alpha''' = -b y(t) - a^{3} \xi(t).\] The exterior
product of these vectors give us
\[V(t) =
  \left\vert
    \begin{matrix}
      x(t) & y(t) & -\eta(t) & -\xi(t) \\
      0   & -b    &  0  &-a \\ 
      b    & 0    &  a^{2} & 0 \\
      0    & b    &  0  & a^{3} 
    \end{matrix}
  \right\vert = (ba^{3}(1 - a^{2}) x(t) + (b^{2}a(1 - a^{2}) \eta(t).\]
obtaining the third normal vector$V(t)=\sqrt{b^{2} - a^{4}}N_3 = a^{2} x(t) + b \xi(t)$.

\begin{example}
  For $0 < b < a < \sqrt{b}$. Given the positive helix
  \[\alpha(s) = \left(b \cos \frac{s}{\sqrt{a^{2} - b^{2}}},b \sin \frac{s}{\sqrt{a^{2} - b^{2}}}, 
      \cos \frac{as}{\sqrt{a^{2} - b^{2}}}, \sin \frac{as}{\sqrt{a^{2} -
          b^{2}}} \right),\] and taking the positive field given by its
  third normal
  \[N_3(s) = \frac{1}{\sqrt{b^{2} - a^{4}}}\left(-a^{2} \sin t,a^{2}
      \cos t, -b \sin a t, b \cos a t\right),\] the parametric surface
  given by
  \[f(w) = \mathfrak{Re} \int_{0}^{w} (T(\xi) - i N_3(\xi)) d\xi\] is a
  solution of the Problem~\ref{prob:bjorling} such that the curve
  $\alpha(s) = f(s,0)$ is a geodesic of $f(M)$.
\end{example}

Ou last example is an isoclinic ``bi-helix'':

\begin{example} Let $b > 0$ be and, let $m$ and $n$ be two distinct
  natural numbers. Considerer the negative curve
  \[c(t) = \frac{\sqrt{1 + b^{2}}}{m}(\cos mt,\sin mt,0,0) +
    \frac{b}{n}(0,0,\cos nt,\sin nt ).\] The parametric isoclinic
  minimal surface is given by
  \[f(w)=c(0) + \int_{0}^{w} (c'(\xi) - iL(c'(\xi))) d \xi,\] where L is
  the linear operator given in (\ref{trans}).
\end{example}

\end{document}